\documentclass[12pt]{article}
\usepackage[utf8]{inputenc}
\usepackage[T1]{fontenc}
\usepackage{amsmath,amssymb,amsthm}
\usepackage{mathrsfs}
\usepackage[danish,english]{babel}
\usepackage[T1]{fontenc}
\usepackage[pass]{geometry}
\usepackage{graphicx}
\usepackage[usenames,dvipsnames]{color}
\usepackage{eucal}
\usepackage[all]{xy}
\usepackage{tikz}
	\usetikzlibrary{matrix}
\usepackage{latexsym}
\usepackage{lipsum}
\usepackage{color}

\theoremstyle{plain} 
	\newtheorem{thm}{Theorem}[section]
	\newtheorem{lem}[thm]{Lemma}
	\newtheorem{cor}[thm]{Corollary}
	
	\newtheorem{prop}[thm]{Proposition}
\theoremstyle{definition} 
	\newtheorem{defi}[thm]{Definition}

	\newtheorem{example}[thm]{Example}

\theoremstyle{remark}
	\newtheorem{remark}[thm]{Remark}

\newcommand{\N}{\mathbb{N}}
\renewcommand{\O}{\mathbb{O}}

\newcommand{\Q}{\mathbb{Q}}
\newcommand{\R}{\mathbb{R}}

\renewcommand{\d}{\mathcal{D}}

\newcommand{\p}{\mathcal{P}}

\newcommand{\s}{\mathcal{S}}

\renewcommand{\epsilon}{\varepsilon}
\renewcommand{\phi}{\varphi}

\newcommand{\ol}[1]{\overline{#1}}

\newcommand{\set}[1]{ \left\lbrace #1\right\rbrace}

\usepackage[backrefs]{amsrefs}

\newcommand{\diam}{\text{diam}}
\newcommand{\omf}{\omega^{2}}
\newcommand{\cw}{\curlywedge}

\newcommand{\pc}{\prec}
\newcommand{\bap}{ \text{B}_\p}
\newcommand{\cbap}{ \ol{\text{B}}_\p}
\newcommand{\bas}{ \text{B}_\s}
\newcommand{\cbas}{  \ol{\text{B}}_\s}
\newcommand{\ba}{ \text{B}}

\newcommand{\cb}{ \text{CB}}

\usepackage{tikz}
\tikzset{node distance=2cm, auto}
\usepackage[all]{xy}

\usepackage[capitalize]{cleveref}

\begin{document}
\title{\bf Cantor-Bendixson type ranks on Polish spaces}
\author{Vibeke Quorning}
\maketitle

\begin{abstract} For any Polish space $X$ it is well-known that the Cantor-Bendixson rank provides a co-analytic rank on $F_{\aleph_0}(X)$ if and only if $X$ is a $\sigma$-compact. In the case of $\omega^\omega$ one may recover a co-analytic rank on $F_{\aleph_0}(\omega^\omega)$ by considering the Cantor-Bendixson rank of the induced trees instead. In this paper we will generalize this idea to arbitrary Polish spaces and thereby construct a family of co-analytic ranks on $F_{\aleph_0}(X)$ for any Polish space $X$. We study the behaviour of this family and compare the ranks to the original Cantor-Bendixson rank.
The main results are characterizations of the compact and $\sigma$-compact Polish spaces in terms of this behaviour. 
\end{abstract}

\section*{Introduction}
In this paper we will consider the Effros Borel space $F(X)$ of closed subsets of a Polish space $X$, and the co-analytic subset $F_{\aleph_0}(X)$ of countable closed subsets. Recall that a subset is called co-analytic if it is the complement of an analytic set. It is well known that $F_{\aleph_0}(X)$ is co-analytic  and not Borel when $X$ is uncountable. 

A key property of co-analytic sets is that they admit a co-analytic rank into $\omega_1$. Given a set $A$, a rank into $\omega_1$ is a map $\phi\colon A\rightarrow \omega_1$. If $A$ is a co-analytic subset of a Polish space $X$, the rank $\phi$ is co-analytic if for all $\alpha<\omega_1$ the initial segment
\[ A_\alpha^\phi=\set{x\in A\mid \phi(x)\leq\alpha}\]
is Borel in a uniform manner in $X$.
This in particular ensures that any co-analytic set $A$ is an increasing union of $\omega_1$ many Borel sets $(B_\alpha)_{\alpha\in \omega_1}$\[A=\bigcup_{\alpha<\omega_1}B_\alpha.\]

The main result concerning co-analytic ranks is the Boundedness Theorem, which states that if $\phi\colon A\rightarrow \omega_1$ is a co-analytic rank from a co-analytic subset $A\subseteq X$, then  $A$ is Borel in $X$ if and only if $$\sup\set{\phi(x)\mid x\in A}<\omega_1.$$ Moreover, if $B\subseteq A$ is analytic in $X$, then $\sup\set{\phi(x)\mid x\in B}<\omega_1$.

The first part of the theorem highlights how these ranks provides a powerful tool for proving that certain subsets are not Borel. For example in \cite{WW} co-analytic ranks are used to prove that, in a certain parametrization of countable groups, the subset of elementary amenable groups is not Borel while the subset of amenable groups is. This result thereby gives a non-constructive existence proof of an amenable group that is not elementary amenable. 

The second part of the theorem ensures a uniformity of the co-analytic ranks that a fixed co-analytic set $A$ admits. Indeed it implies that if $\phi,\psi \colon A\rightarrow \omega_1$ are both co-analytic ranks, then there exist  functions $f,g\colon \omega_1\rightarrow \omega_1$ such that $\phi(x)\leq f(\psi(x))$ and $\psi(x)\leq g(\phi(x))$ for all $x\in A$. So all co-analytic ranks on $A$ agree on which subsets are bounded and which subsets are not. Therefore, each co-analytic subset of a Polish space admits a natural $\sigma$-ideal of bounded sets. 
Even though it is known that any co-analytic set has a co-analytic rank and therefore also admits this $\sigma$-ideal, the proof does not provide a concrete rank for a given co-analytic set. Thus it is of interest to find concrete co-analytic ranks to determine the structure of the co-analytic set in terms of the $\sigma$-ideal of bounded sets. 

As mentioned earlier, we will here consider $F_{\aleph_0}(X)$. A natural rank on this co-analytic set is the Cantor-Bendixson rank, which assigns to each $F\in F_{\aleph_0}(X)$ the length of the transfinite process of removing isolated points. However, the Cantor-Bendixson rank is only co-analytic when the underlying Polish space $X$ is $\sigma$-compact. In fact, as mentioned in  \cite[Section 34.F]{Kechris}, there does not seem to be known an explicit co-analytic rank on $F_{\aleph_0}(X)$ for a general Polish space $X$.

In the specific case of the Baire space $\omega^\omega$, there is a natural correspondence between $F(\omega^\omega)$ and the trees on $\omega$. A tree $T$ on $\omega$ is a subset of finite sequences of numbers in $\omega$, which is closed under initial segments.  One then associates to each $F\in F(\omega^\omega)$ the tree $T_F$ consisting of all finite initial segments of the elements of $F$. 
Moreover, on the set of trees on $\omega$ there is a Cantor-Bendixson rank, which assigns to each tree the length of the transfinite process of removing isolated branches of the tree (see \cite[Exercise 6.15]{Kechris}). We obtain a co-analytic rank on $F_{\aleph_0}(\omega^\omega)$  by assigning to each $F$ the Cantor-Bendixson rank of $T_F$.  

The first goal of this paper is to generalize the construction used in the  case  of the Baire space to arbitrarily Polish spaces. This is done by using a countable family of balls, induced by a complete metric and a countable dense sequence of the Polish space, to encode the closed subsets. A presentation $\p=(X,d,(x_i)_{i\in \omega})$ of a Polish space $X$ is a Polish space $X$ equipped with a fixed choice of a complete compatible metric and a countable dense sequence. We then prove the following theorem. 

\begin{thm}\label{01}
For each presentation $\p=(X,d,(x_i)_{i\in \omega})$ of a Polish space $X$ there is a co-analytic rank $\phi_\p\colon F_{\aleph_0}(X)\rightarrow \omega_1$ induced by the countable family
$$\set{B_d(x_i,2^{-k})\mid i,k\in \omega}$$ of $\p$-balls.
\end{thm}
The above theorem implies that for each Polish space $X$ we have a potentially huge family of co-analytic ranks on $F_{\aleph_0}(X)$. The second goal of this paper is to investigate how this family of co-analytic ranks behaves. We prove results stating how the chosen presentation affects the ranks one obtains, and how they relate to the original Cantor-Bendixson rank on $F_{\aleph_0}(X)$. The main results are the following characterizations of the compact and $\sigma$-compact Polish spaces in terms of this behaviour.

Below we will use the following notation. For a Polish space $X$, we let $|X|_{K_\sigma}$ denote the length of the transfinite process of removing locally compact points. In the case where $X$ is $\sigma$-compact, this process terminates with the empty set. Thus in this case $|\cdot |_{K_\sigma}$  measures how far from compact the $\sigma$-compact Polish space is. Moreover, for $F\in F_{\aleph_0}(X)$, we let $|F|_{\cb}$ denote the Cantor-Bendixson rank of $F$.
\begin{thm} Let $X$ be a Polish space. The following are equivalent:
\begin{itemize}
\item[1)] $X$ is compact.
\item[2)] The family $(\phi_\p)_\p$, where $\p$ varies over all presentations of $X$, is uniformly bounded by the Cantor-Bendixson rank.
\item[3)] For any presentation $\p$ of $X$ we have $$\phi_\p(F)<\omega|F |_{\cb}$$ for all $F\in F_{\aleph_0}(X)$.
\end{itemize}
\end{thm}

\begin{thm} Let $X$ be a Polish space. The following are equivalent:
\begin{itemize}
\item[1)] $X$ is $\sigma$-compact.
\item[2)] For some presentation $\p$ of $X$ there exists $f\colon \omega_1\rightarrow \omega_1$ such that $\phi_\p(F)\leq f(|F|_{\cb})$ for all $F\in F_{\aleph_0}(X)$.
\item[3)] For each presentation $\p$ there exist ordinals $\alpha_\p,\beta_\p <\omega_1$ such that 
\[\phi_\p(F)\leq (\omega|F|_{\cb} + \alpha_\p)|X|_{K_\sigma} + \beta_\p,\]
for all $F\in F_{\aleph_0}(X)$. 
\end{itemize}
\end{thm}

The paper is organized as follows. The first section consists of various preliminaries and includes a very brief introduction to co-analytic ranks. In the second section we describe the construction of the ranks, and in the third section we investigate the dependence of the rank on the chosen presentation. In the fourth and final section we compare the family of obtained ranks to the Cantor-Bendixson rank and prove the characterizations described above. 

\medskip
{\it Acknowledgments.} I would like to thank my advisor, Asger T\"ornquist, for many valuable discussions. I would also like to thank Joshua Hunt and Alexander S. Kechris for several useful comments on  earlier drafts.

The author was partially supported by Lars Hesselholt's Niels Bohr Professorship.

\section{Preliminaries}
A \textbf{Polish space} is a separable completely metrizable topological space. 
Examples of these include $\R$ and $\omega^\omega$, where the latter is equipped with the product topology induced by the discrete topology on $\omega=\set{0,1,2,\ldots }$. 
A subset of a Polish space is called \textbf{analytic} if it is the continuous image of a Polish space. The complement of an analytic set is called \textbf{co-analytic}. For an uncountable Polish space both the family of analytic sets and the family of co-analytic sets contain the Borel sets as a proper subset. In fact, by Souslin's Theorem the Borel sets are exactly the sets that are both analytic and co-analytic. Both classes are also closed under Borel pre-images.

A \textbf{standard Borel space} is a Polish space equipped with its Borel $\sigma$-algebra. An important example of such a space is the \textbf{Effros Borel space}, which is defined as follows. Let $X$ be a Polish space and let $F(X)$ denote the set of all closed subsets of $X$. We equip $F(X)$ with the Borel structure generated by the sets
\[\set{F\in F(X)\mid F\cap U\neq \emptyset },\]
where $U\subseteq X$ varies over all open subsets, which turns $F(X)$ into a standard Borel space. We are in particular interested in two subsets of the Effros Borel space, namely the subset $K(X)$ consisting of all the compact subsets and the subset $F_{\aleph_0}(X)$ consisting of all countable closed subsets. 
The former is Borel, while the latter is a non-Borel co-analytic subset when $X$ is uncountable. 
\\

A \textbf{rank} on a set $S$ is a map $\phi\colon S\rightarrow \O\N$. Here $\O\N$ denotes the ordinals. We will let $\alpha_\phi=\sup\set{ \phi(x)\mid x\in S}$. If $\phi(S)=\alpha_\phi$ the rank is called \textbf{regular}. For each rank $\phi\colon S\rightarrow \O\N$ we obtain a prewellordering $\leq_\phi$ on $S$ given by
\[x\leq_\phi y \iff \phi(x)\leq \phi(y).\]
For the class of co-analytic sets certain ranks are of special interest, as these ensure that this prewellordering has nice definability properties.
\begin{defi} Let $X$ be a Polish space and $A\subseteq X$ a co-analytic set. A rank $\phi\colon A\rightarrow \O\N$ is said to be co-analytic if there exist binary relations $R_\phi^{\Pi_1^1},R_\phi^{\Sigma_1^1}\subseteq X\times X$ such that $R_\phi^{\Pi_1^1}$ is co-analytic, $R_\phi^{\Sigma_1^1}$ is analytic and for any $y\in A$ we have
\[xR_\phi^{\Pi_1^1}y \iff xR_\phi^{\Sigma_1^1}y \iff (x\in A )\wedge (x\leq_\phi y)\]
for all $x\in X$.
\end{defi}

So a rank is co-analytic if the induced prewellordering on $A$ extends to both an analytic and a co-analytic relation on $X$, which preserve the initial segments of $A$. This implies in particular that for each $\alpha<\alpha_\phi$ we have
\[A_\alpha^\phi =\set{x\in A\mid \phi(x)\leq \alpha}\]
is Borel in $X$. 
The following well-known theorem ensures that there always exists a co-analytic rank on any co-analytic set with $\alpha_\phi\leq \omega_1$. For a proof see Theorem 34.4 in \cite{Kechris}.

\begin{thm}\label{existencecorank} Let $X$ be Polish and $A\subseteq X$ a co-analytic set. There exists a co-analytic rank $\phi\colon A\rightarrow \omega_1$.
\end{thm}
The main interest in these co-analytic ranks comes from the Boundedness Theorem. A proof can be found in \cite[Theorem 35.23]{Kechris}.
\begin{thm}[The Boundedness Theorem]\label{boundednessthm} Let $X$ be Polish, $A\subseteq X$ a co-analytic set and $\phi\colon A\rightarrow \omega_1$ a co-analytic rank. Then the following hold:
\begin{itemize}
\item[1)] $A$ is Borel if and only if $\alpha_\phi<\omega_1$.
\item[2)] If $B\subseteq A$ is analytic, then $\sup\set{\phi(x)\mid x\in B}<\omega_1$.
\end{itemize}
\end{thm}

Note that it follows directly from this theorem that if $\phi,\psi\colon A\rightarrow \omega_1$ are both co-analytic ranks, then there exists $f\colon \omega_1\rightarrow \omega_1$ such that
$\phi(x)\leq f(\psi(x))$ for all $x\in A$. Indeed, put $f(\alpha)=\sup\set{\phi(x)\mid \psi(x)\leq \alpha}$ for all $\alpha<\omega_1$. In this case we say that $\phi$ is \textbf{bounded} by $\psi$ via $f$. If we have a family of ranks $(\phi_i)_{i\in I}$  and there is a function $f\colon \omega_1\rightarrow \omega_1$ such that $\phi_i(x)\leq f(\psi(x))$ for all $x\in A$ and all $i\in I$, we say that the family is \textbf{uniformly bounded} by $\psi$ via $f$. \\
 
One way to obtain co-analytic ranks on co-analytic subsets of the Effros Borel space is done by use of Borel derivatives. Let $X$ be Polish and let $\d= F(X)$ or $\d=K(X)$. A Borel derivative on $\d$ is a Borel map $D\colon \d\rightarrow \d$ satisfying $D(F)\subseteq F$ and $F_0\subseteq F_1\implies D(F_0)\subseteq D(F_1)$ for all $F,F_0,F_1\in \d$. Whenever we have such a derivative, we define the iterated derivatives of $F\in \d$ as follows:
\[D^0(F)=F, \qquad D^{\alpha+1}(F)=D(D^\alpha(F)) \quad \ \text{and} \ \quad  D^\lambda(F) =\bigcap_{\beta<\lambda} D^\beta(F),\]
where $\alpha,\lambda\in \O\N$ and $\lambda$ is a limit ordinal. 
The least $\alpha\in \O\N$ satisfying $D^\alpha(F)=D^{\alpha+1}(F)$ is denoted by $|F|_D$ and we let $D^\infty(F)=D^{|F|_D}(F)$. Thus we obtain a rank on $\d$ given by $F\mapsto |F|_D$. 

An important example of such a derivative is the \textbf{Cantor-Bendixson derivative} 
$D_{\cb}\colon F(X)\rightarrow F(X)$ which is given by
\[D_\cb(F)=\set{x\in F\mid x \ \text{is a limit point in} \ F}.\]
Note that  $|F|_\cb<\omega_1$ for all $F\in F(X)$. We call $|F|_\cb$ the \textbf{Cantor-Bendixson rank} of $F$.  To simplify notation later on we will write $F^\alpha$ instead of $D_\cb^\alpha(F)$ for all $F\in F(X)$ and all $\alpha<\omega_1$. 
We should also point out that $D_\cb^\infty(F)=P_F$, where $P_F$ denotes the perfect kernel of $F$.
Recall that a subspace of a Polish space is called \textbf{perfect} if it is closed and does not contain any isolated points. Note that if such a set is non-empty, then it is uncountable. Every Polish space $X$ decomposes uniquely as
$X=P\sqcup C$, where $P$ is perfect and $C$ is countable (see \cite[Theorem 6.4]{Kechris}). The subset $P$ is called the \textbf{perfect kernel} of $X$. Thus $D^\infty(F)=\emptyset$ if and only if $F\in F_{\aleph_0}(X)$.

The theorem below provides a sufficient criterion for when a rank induced by a Borel derivative is co-analytic. For a proof see Theorem 34.10 in \cite{Kechris}.

\begin{thm}\label{co-anrankfx} Let $X$ be a Polish space. Assume $D\colon \d \rightarrow \d$ is a Borel derivative, where either $\d=K(X)$, or $X$ is $\sigma$-compact and $\d=F(X)$. Then
\[\Omega_D=\set{F\in \d \mid D^\infty(F)=\emptyset}\]
is a co-analytic set and the rank $\phi_D\colon \Omega_D\rightarrow \O\N$ given by $\phi_D(F)=|F|_D$ is co-analytic.
\end{thm}

Therefore, when $X$ is a $\sigma$-compact Polish space the rank $F\mapsto |F|_\cb$ is co-analytic on $F_{\aleph_0}(X)$. However, the subset
\[\set{F\in F_{\aleph_0}(\omega^\omega)\mid F \ \text{is discrete}}=\set{F\in F_{\aleph_0}(\omega^\omega)\mid |F|_{\cb}\leq 1}\]
is not Borel, and hence in this case the rank cannot be co-analytic. In fact, since every Polish space which is not $\sigma$-compact contains a closed copy of $\omega^\omega$, this implies that the Cantor-Bendixson rank is co-analytic on $F_{\aleph_0}(X)$ if and only if $X$ is $\sigma$-compact.

\section{The construction}
 
Let $X$ be a Polish space. A  \textbf{presentation} of $X$ is a triple $(X,d,\ol{x})$, where $d$ is a complete compatible metric on $X$ and $\ol{x}=(x_n)_n$ is a dense sequence in $X$. In this section we will construct a co-analytic rank on $F_{\aleph_0}(X)$ from a presentation $(X,d,\ol{x})$ of $X$. As we will see, this rank will share many of the properties of the standard Cantor-Bendixson rank.\\

Fix a presentation $\p=(X,d,\ol{x})$ of  $X$. For each $n\in \omf$ we let \[\bap(n)=\set{x\in X\mid d(x,x_{n(0)})<2^{-n(1)-1}}\] and \[\cbap(n)=\set{x\in X\mid d(x,x_{n(0)})\leq 2^{-n(1)-1}}.\] 
In general for $x\in X$ and $r>0$ we will let $B_d(x,r)$ and $\ol{B}_d(x,r)$ denote the open and closed ball around $x$ with radius $r$ with respect to $d$.
 Define binary relations $\prec_\p$ and $\cw_\p$ on $\omega^2$ by
\begin{align*}
n \prec_\p m &\iff  n(1)<m(1) \ \text{and} \ \cbap(m) \subseteq \bap(n)   \\
\\
n\cw_\p m &\iff \cbap(n) \cap \cbap(m) = \emptyset
\end{align*}
for all $n,m\in \omega^2$. We say that a subset $A\subseteq \omf$ is \textbf{$\p$-closed} if for any $n\in A$ and $m\in \omega^2$ with $m\prec_\p n$, we have that $m\in A$.
We say that $A$ is \textbf{$\p$-perfect} if for all $n\in A$ there are $u,v\in A$ such that $n\prec_\p u,v$ and $u\cw_\p v$. Lastly, we say that $A$ is  \textbf{$\p$-pruned} if for all $n\in A$ there is $m\in A $ with $n\pc_\p m$. Clearly any $\p$-perfect subset is also $\p$-pruned.

We will now establish a definable correspondence between the closed subsets of $X$ and the subsets of $\omf$ that are $\p$-closed and $\p$-pruned. First, for $F\in F(X)$ we let
\[A_F^\p=\set{ n\in \omf \mid \bap(n)\cap F \neq \emptyset}.\]
It is easily seen that $A_F^\p$ is $\p$-closed and $\p$-pruned for all $F\in F(X)$ and that the map $F\mapsto A_F^\p$ from $F(X)$ to $P(\omf)$ is Borel. For the other direction, consider the set of $\p$-compatible sequences in $\omf$ given by
\[ [\omf ]_\p = \set{ (n_i)_i\in (\omf)^\omega \mid (\forall i\in \omega) \ n_i\pc_\p n_{i+1} }.\] 
Note that $[\omf]_\p\subseteq (\omega^2)^\omega$ is closed and hence Polish.
Since $d$ is a complete metric, we obtain a surjective continuous map $\pi_\p \colon [\omf]_\p \rightarrow X$ given by
\[ \pi_\p((n_i)_i)=x \iff \bigcap_{i\in \omega}\bap(n_i)=\set{x}.\]
For any $A\in P(\omf)$ we let 
\[ [A]_\p = \set{ (n_i)_i\in [\omf]_\p \mid (\forall i\in \omega) \ n_i \in A }\]
and let $F_A^\p=\pi_\p([A]_\p)$. If $A$ is $\p$-closed, then $F_A^\p$ is closed in $X$. \\

The following proposition is now straightforward to check.
\begin{prop}\label{correspondence} Let $\p =(X,d,\ol{x})$ be a presentation of a Polish space $X$. Then the following hold:
\begin{itemize}
\item[1)] The map $F\mapsto A_F^\p$ is a Borel isomorphism from $F(X)$ to the space of $\p$-pruned and $\p$-closed subsets of $\omf$ with inverse $A\mapsto F_A^\p$.
\item[2)] For any $F\in F(X)$ it holds that $F$ is perfect if and only if $A_F^\p$ is $\p$-perfect.
\end{itemize}
\end{prop}

To finish the construction we will define a derivative on $P(\omf)$ and use the correspondence given in Proposition \ref{correspondence} to obtain a rank on $F(X)$, which is co-analytic when restricted to $F_{\aleph_0}(X)$.
 
 Define $D_\p\colon P(\omf) \rightarrow P(\omf)$ by 
\[D_\p(A)=\set{ n\in A \mid (\exists u,v\in A) \ n\prec_\p u, v \ \text{and} \ u\cw_\p v}.\]
 It is clear that $D_\p(A)\subseteq A$ and that $B\subseteq A$ implies $D_\p(B)\subseteq D_\p(A)$ for all $A,B\in P(\omf)$, hence $D_\p$ is a derivative on $P(\omf)$. Thus for each $A\in P(\omf)$ there is a least ordinal $\alpha <\omega_1$ such that $D_\p^\alpha(A)=D_\p^{\alpha+1}(A)$. We let $|A|_\p$ denote this least ordinal. 
 
It remains to prove the following theorem.  
 
\begin{thm}\label{RankDef} Let $\p=(X,d,\ol{x})$ be a presentation of a Polish space $X$. The map $\phi_\p\colon F_{\aleph_0}(X)\to \omega_1$ given by $F\mapsto |A_F^\p|_\p$ is a co-analytic rank.
\end{thm} 
\begin{proof}
First, since $D_\p$ is continuous, it follows by Theorem \ref{co-anrankfx} that the set
\[ \Omega_\p = \set{ A\in P(\omf) \mid D_\p^{|A|_\p}(A)=\emptyset} \]
is co-analytic and that the map $A\mapsto |A|_\p$ is a co-analytic rank. By the definition of $D_\p$, we have that $D_\p^\alpha(A)$ is $\p$-closed for all $\alpha<\omega_1$. Moreover, $D_\p^{|A|_\p}(A)$ is the largest $\p$-perfect subset of $A$, in the sense that it contains any $\p$-perfect subset of $A$. Therefore, by Proposition \ref{correspondence}, $\pi_\p\left([D_\p^{|A|_\p}(A_F^\p)]_\p\right)$ must be the perfect kernel of $F$ for any $F\in F(X)$. So for any $F\in F(X)$ it holds that
\[F\in F_{\aleph_0}(X) \iff A_F^\p\in \Omega_\p.\] 
Hence, since the map $F\mapsto A_F^\p$ from $F(X)$ to $P(\omf)$ is Borel, we may conclude that $\phi_\p$ is a co-analytic rank on $F_{\aleph_0}(X)$.
\end{proof}

Since every Polish space $X$ admits many presentations, we obtain a huge family,
\[ \set{\phi_\p \mid \p \ \text{is a presentation of} \ X},\]
of co-analytic ranks on $F_{\aleph_0}(X)$. In the rest of this paper we will investigate the behaviour of this family for different classes of Polish spaces. 
\section{Dependence on presentation}
In the previous section we constructed a co-analytic rank $\phi_\p \colon F_{\aleph_0}(X)\rightarrow \omega_1$ for each presentation $\p=(X,d,\ol{x})$ of a Polish space $X$ . In this section, we investigate the extend to which the rank depends on the chosen presentation. We divide this investigation in two parts. 

First we consider the variations that occur when varying the dense sequence while holding the metric fixed. We isolate a class of Polish spaces for which the construction is completely independent of the dense sequence, and a broader class for which the ranks agree up to one step. In general the ranks can change more significantly, but we will recover a bound on this change.

Next we instead consider the variations that occur when varying the metric while fixing the dense sequence. It is clear that even a small change in the chosen metric can affect the induced rank. We will see that even for the discrete countable Polish space there exist presentations for which the rank of the whole space varies arbitrarily. We will also find a bound for the variation in the case where the two metrics are equivalent (in the strong sense).\\

A useful tool to compare the ranks induced by these presentations is the following simple lemma. The proof is straightforward.

\begin{lem}\label{homoless} Let $\p=(X,d,\ol{x})$ and $\s=(Y,\delta,\ol{y})$ be presentations of the Polish spaces $X$ and $Y$, respectively. If $A\subseteq \omf$ and $\psi \colon A\rightarrow \omf$ satisfies
\[n\pc_\p m \implies \psi(n)\pc_\s \psi(m)\quad \text{and} \quad n\cw_\p m\implies \psi(n)\cw_\s \psi(m)\]
for all $n,m\in A$, then $|A|_\p\leq |\psi(A)|_\s$.
\end{lem}

 Note that $|A|_\p=|\psi(A)|_\s$ if we have bi-implications instead of implications in the lemma above.

\subsection{Change of dense sequence}
First we investigate what happens when we change the dense sequence. The objective is to isolate classes of Polish metric spaces for which the induced rank does not depend on the chosen sequence. 

\begin{prop}\label{indepdenseultra} Let $X$ be a Polish space and $d$ a compatible complete ultrametric on $X$. Then any pair of presentations $\p=(X,d,\ol{x})$ and $\s=(X,d,\ol{y})$ satisfies  $\phi_\p=\phi_\s$.
\end{prop}

\begin{proof} First, note that, since $d$ is an ultrametric, we have
\[ \set{ \bap((i,j)) \mid i\in \omega } = \set{ \bas((i,j)) \mid i\in \omega}\]
for each $j\in \omega$. Therefore we can choose $f\colon \omega \to \omega$ such that $\bap((i,j))=\bas((f(i),j))$ for all $i\in \omega$, and define $\psi \colon \omf \to \omf$ by $\psi(n)=(f(n(0)),n(1))$. Note that $n\pc_\p m \iff \psi(n)\pc_\s \psi(m)$ and $n\cw_\p m \iff \psi(n)\cw_\s \psi(m)$ for all $n,m\in \omf$. 
Now let $F\in F_{\aleph_0}(X)$. Since $\psi(A_F^\p)\subseteq A_F^\s$, it follows by Lemma \ref{homoless} that $\phi_\p(F)\leq \phi_\s(F)$. By symmetry we conclude $\phi_\p=\phi_\s$, as wanted.

\end{proof}

Next we will see that for compact Polish spaces the ranks only depend on the chosen metric. We will need the following analogue of Königs Lemma. 
 
\begin{lem}\label{Konig} Let $X$ be compact, $\p=(X,d,\ol{x})$ a presentation of $X$ and let $A\subseteq \omf$ be $\p$-closed. If there is $(n_k)_k\in A$ such that $n_k(1)\rightarrow \infty$ as $k\rightarrow \infty$, then $[A]_\p\neq \emptyset$.
\end{lem}

\begin{proof}
Assume there is $(n_k)_k\in A$ such that $n_k(1)\rightarrow \infty$ as $k\rightarrow \infty$, and let $z_k\in \bap(n_k)$ for all $k\in \omega$. Then, by compactness of $X$, there is $z\in X$ and a subsequence $(z_{k_i})_i\subseteq (z_k)_k$ such that $z_{k_i}\rightarrow z$ as $i\rightarrow \infty$. 
To prove $z\in [A]_\p$, assume that $z\in \bap(m)$ for some $m\in \omf$. Then there is $\epsilon >0$ and $N\in \omega$ such that $B_d(z,\epsilon)\subseteq \bap(m)$ and $z_{k_i}\in B_d(z,\epsilon/3)$ for all $i\geq N$. Moreover, as $\diam_d(\bap(n_k))\rightarrow 0$ as $k\rightarrow \infty$, we may choose $M\geq N$ such that $\diam_d(\bap(n_{k_j}))<\epsilon /3$ for all $j\geq M$. Then $m\pc n_{k_j}$ for any $j\geq M$ and hence, since $A$ is $\p$-closed, we conclude that $m\in A$.
\end{proof}

\begin{remark}\label{remkonig} The argument in the proof of Lemma \ref{Konig} can also be used to prove the following statement. Let $\p$ be a presentation of any Polish space $X$, let $F\in F_{\aleph_0}(X)$ and $n\in A_F^\p$ with $\ol{\bap(n)\cap F}$ compact. Assume $(n_k)_k\in A_F^\p$ satisfy that $n\pc_\p n_k$ for each $k\in \omega$ and that $n_k(1)\rightarrow \infty$ as $k\rightarrow \infty$. Then there exists $x\in \ol{\bap(n)\cap F}$ such that for any $U\subseteq X$ open with $x\in U$ there is $N\in \omega$ such that $\cbap(n_k)\subseteq U$ for all $k\geq N$.

\end{remark}

We will now prove that for a class of Polish spaces, which includes the compact Polish spaces, the rank only depends on the chosen metric. We call a metric \textbf{proper} if all the closed balls are compact. 
\begin{thm}\label{indepdenseHB} Let $X$ be a Polish space and $d$ a complete compatible proper metric on $X$. Then any pair of presentations $\p=(X,d,\ol{x})$ and $\s=(X,d,\ol{y})$ satisfy  $\phi_\p=\phi_\s$.
\end{thm}
\begin{proof} 
Let $F\in F_{\aleph_0}(X)$. First, we prove by induction on $\alpha<\omega_1$ that if $n\in D_\p^\alpha(A_F^\p)$ and $\epsilon>0$ , then there is $m\in D^\alpha_\s(A_F^\s)$ such that \[\cbas(m)\subseteq B_d( x_{n(0)},2^{-n(1)-1}+\epsilon)\]
and $m(1)\geq n(1)$.

For $\alpha = 0$ the statement is easily seen to be true. 

Assume the statement holds for some $\alpha<\omega_1$. Let $n\in D^{\alpha+1}_\p(A_F^\p)$ and let $\epsilon>0$ be given. Then there is $u,v\in D_\p^\alpha(A_F^\p)$ such that $n\pc_\p u,v$ and $u\cw_\p v$. The compactness of $\cbap(n)$ now implies that there exists $\rho>0$  such that 
\[B_d(x_{u(0)},2^{-u(1)-1} + \rho),B_d(x_{v(0)},2^{-v(1)-1} + \rho)\subseteq B_d(x_{n(0)},2^{-n(1)-1} - \rho)\] and \[B_d(x_{u(0)},2^{-u(1)-1} + \rho)\cap B_d(x_{v(0)},2^{-v(1)-1} + \rho)=\emptyset.\] 
By the induction hypothesis there is $\tilde{v},\tilde{u}\in D_\s^\alpha(A_F^\s)$ such that
\[\cbas(\tilde{u})\subseteq B_d(x_{u(0)},2^{-u(1)-1}+\rho),\qquad \cbas(\tilde{v})\subseteq B_d(x_{v(0)},2^{-v(1)-1}+\rho),\]
$\tilde{u}(1)\geq u(1)$ and $\tilde{v}(1)\geq v(1)$.
Now choose $m\in \omega^2$ such that $d(x_{n(0)},y_{m(0)})<\min\set{\epsilon, \rho}$ and  $m(1)=n(1)$. Then $m\pc_\s \tilde{u},\tilde{v}$ and $\tilde{u}\cw_\s \tilde{v}$. Therefore $m\in D_\s^{\alpha+1}(A_F^\s)$, as wanted. 

Finally, assume that the statement holds for all $\beta <\lambda$ for some limit ordinal $\lambda <\omega_1$.
Let $n\in D_\p^\lambda(A_F^\p)$ and let $\epsilon>0$ be given. Now choose $(\beta_k)_k\in \lambda$ such that $\bigcup_{k\in \omega}\beta_k=\lambda$. Then there must exist $(n_k)_k\in A_F^\p$ satisfying $n\pc_\p n_k$, $n_k\in D^{\beta_k}_\p(A_F^\p)$ and $n_k(1)\rightarrow \infty$ as $k\rightarrow \infty$.  Therefore we may let $z\in \ol{\bap(n)\cap F}$ satisfy the statement of Remark \ref{remkonig}. Now choose $m\in \omega^2$ such that $m(1)\geq n(1)$, $2^{-m(1)-1}<\epsilon/2$ and $z\in \bas(m)$. Then there is $N\in \omega$ such that $\cbap(n_k)\subseteq \bas(m)$ for all $k\geq N$. So, by use of compactness of $\cbas(m)$ and the induction hypothesis as above, we deduce that there exist $m_k\in D^{\beta_k}_\s(A_F^\s)$ with $m\pc_\s m_k$ for all $k\geq N$. Therefore we have $m\in D^\lambda(A_F^\s)$.\\

We have now established that if for some $\alpha<\omega_1$ we have $n\in D_\p^\alpha(A_F^\p)$, then there is $m\in D^\alpha_\s(A_F^\s)$. Therefore $D^\alpha_\s(A^\s_F)=\emptyset$ implies $D_\p^\alpha(A^\p_F)=\emptyset$ for all $\alpha<\omega_1$. By symmetry we conclude $\phi_\p=\phi_\s$.

\end{proof}

The next proposition isolates the two most important properties used in the proof above. 

\begin{prop}\label{indepdensecriter} Let $X$ be a Polish space and $d$ a complete compatible metric on $X$ such that for all $x,y\in X$ and $\epsilon,\xi>0$ we have that
\begin{itemize}
\item[1)] $\ol{B}_d(x,\epsilon)\subseteq B_d(y,\xi)$ implies that there exists $\rho>0$ such that $$B_d(x,\epsilon+\rho)\subseteq B_d(y,\xi-\rho).$$
\item[2)] $\ol{B}_d(x,\epsilon)\cap \ol{B}_d(y,\xi)=\emptyset$ implies that there exists $\rho>0$ such that $$ B_d(x,\epsilon+\rho)\cap B_d(y,\xi + \rho)=\emptyset.$$
\end{itemize}
For any pair of presentations $\p=(X,d,\ol{x})$ and $\s=(X,d,\ol{y})$ we have $\phi_\p(F)\leq \phi_\s(F)+1$ for all $F\in F_{\aleph_0}(X)$.
\end{prop}

\begin{proof}Let $F\in F_{\aleph_0}(X)$. As before, we prove by induction on $\alpha<\omega_1$ that for all $n\in D^{\alpha+1}_\p(A_F^\p)$ and $\epsilon>0$ there is $m\in D^\alpha_\s(A_F^\s)$ with $\cbas(m)\subseteq B_d(x_{n(0)},2^{-n(1)-1}+\epsilon)$ and $m(1)\geq n(1)$.

For $\alpha=0$ and the successor case is done exactly as before. So let us briefly argue for the limit case. Assume the statement holds for all $\beta<\lambda$ for some limit $\lambda<\omega_1$. Let $n\in D_\p^{\lambda+1}(A_F^\p)$ and $\epsilon>0$ be given. Then there is $\tilde{n}\in D^\lambda_\p(A_F^\p)$ with $n\pc_\p \tilde{n}$. So there must exist $0<\rho<\epsilon$ such that
\[B_d(x_{\tilde{n}(0)},2^{-\tilde{n}(1)-1}+\rho )\subseteq B_d(x_{n(0)},2^{-n(1)-1}-\rho ).\]
Then, since $\tilde{n}\in D^{\beta+1}_\p(A_F^\p)$ for all $\beta<\lambda$, it follows by the induction hypothesis that there is $m_\beta\in D^\beta_\s(A_F^\s)$ with $m_\beta(1)\geq \tilde{n}(1)$ and
\[\cbas(m_\beta)\subseteq B_d(x_{\tilde{n}(0)},2^{-\tilde{n}(1)-1}+\rho)\subseteq B_d(x_{n(0)},2^{-n(1)-1}-\rho)\]
for all $\beta<\omega_1$. Now choose $m\in \omf$ such that $d(x_{n(0)},y_{m(0)})<\min\set{\epsilon,\rho}$ and $m(1)=n(1)$. Then one easily checks that $\cbas(m_\beta)\subseteq \bap(m)$ and $m(1)<m_\beta(1)$, hence $m\pc_\s m_\beta$ for all $\beta<\omega_1$. So we must have $m\in D^\lambda_\s(A_F^\s)$ as wanted.
\end{proof}
\begin{remark} The proposition above includes all separable Banach spaces and the Urysohn space.
\end{remark}

The following example shows that the conclusion of Proposition \ref{indepdensecriter} is optimal.

\begin{example} Let $H$ be the real infinite dimensional separable Hilbert space. Denote by $(e_i)_i$ an orthonormal basis and by $d$ the metric induced by the inner product.  Consider the presentations $\p=(H,d,\ol{x})$ and $\s=(H,d,\ol{y})$, where $\ol{x}$ is an enumeration of the dense subset
\[D=\set{ \sum_{i\leq n}\lambda_ie_i \mid n\in \omega, \lambda_i\in \Q}\]
and $\ol{y}$ is an enumeration of $D\setminus \set{0}$. 
For each $i\in \omega$ let 
\[F_i\subseteq \set{ \lambda e_i \mid \lambda\in \left[\frac{1}{2}-\frac{1}{2^{i+1}},\frac{1}{2}\right)}\] be finite with $\phi_\s(F_i),\phi_\p(F_i)\geq i$. Then put $F=\bigcup_{i\in \omega}F_i$. We must have $\phi_\s(F)\leq \phi_\p(F)$ and $D_\p^{\omega+1}(A_F^\p)=\emptyset$. Moreover $n\in D_\p^\omega(A_F^\p)$ if and only if $x_{n(0)}=0$ and $n(1)=0$, hence we conclude $\phi_\s(F)=\omega$ and $\phi_\p(F)=\omega+1$.
\end{example}

If we instead consider general Polish spaces, more variation can occur.

\begin{thm}\label{changingsequence}There exist a Polish space $X$,  presentations $\p=(X,d,\ol{x})$, $\s=(X,d,\ol{y})$ of $X$ and $F\in F_{\aleph_0}(X)$ such that $\phi_\p(F)=\omega+1$ and $\phi_\s(F)=2$.
\end{thm}

\begin{proof} First, for $2\leq k<\omega$, we will construct a Polish space $X_k$, presentations $\p_k=(X_k,d_k,\ol{x}_k)$, $\s_k=(X_k,d_k,\ol{y}_k)$ of $X_k$ and $F_k\in F_{\aleph_0}(X_k)$ such that $ \phi_{\p_k}(F)=2+k $ and $\phi_{\s_k}(F)=2$.

Fix $2\leq k <\omega$. Put
\[Y_k=\set{w_i\mid i\in[0,2^{-3}]}\cup\set{z_l\mid l\geq k}\]
and define for each $m\leq k$ a complete metric $\delta_m^k$ on $Y_k$ by
\[\delta_{m}^k(w_i,w_j)=2^{-m-1}|i-j|,\qquad
\delta_{m}^{k}(z_{l},w_i)=2^{-m-1}(1-2^{-l}-i)\]
and
\begin{align*}
\delta_{m}^{k}(z_l,z_n)=\begin{cases} 
0 \qquad \qquad \text{if} \quad l=n\\
2^{-k-1} \qquad \text{if} \quad l\neq n
\end{cases}
\end{align*}
for all $i,j\in [0,2^{-3}]$ and $l,n\geq k$. Note that $\delta_m^k$ induce the same Polish topology on $Y_k$ for all $m\leq k$. 

Next for $s\in 2^{\leq k}$ let 
$Y_k^{s}=\set{y^{s}\mid y\in Y_k}$
denote a copy of $Y_k$ and let $F_k=\set{f_t\mid t\in 2^{k+1}}$. Put
\[X_k=\left(\bigsqcup_{s\in 2^{\leq k}} Y_k^s\right)\sqcup F_k.\]
Define a complete metric $d_k$ on $X_k$ by letting
\begin{itemize}
\item[1)] $d_{k|Y_k^s}=\delta_{\text{ln}(s)}^k$ for all $s\in 2^{\leq k}$.
\item[2)] $d_k(w_0^s,w_0^{s^\smallfrown j})=2^{-\text{ln}(s)-2}+\frac{1}{k+1}2^{-k-3}$ for all $s\in 2^{<k}$ and $j\in 2$.
\item[3)] $d_k(z_l^s,z_l^{s^\smallfrown j})=2^{-\text{ln}(s)-l-2}+\frac{1}{k+1}2^{-k-l-1}$ for all $s\in 2^{<k}$, $j\in 2$ and $l\geq k$.
\item[4)] $d_k(w_0^s,f_{s^\smallfrown j})=2^{-k-1}-2^{-k-3}$ for all $s\in 2^k$ and $j\in 2$.
\end{itemize}
That $d_k$ is indeed a metric on $X_k$ follows from the fact that for each $s\in 2^{<k}$, $j\in 2$, $n\in \omega$ and $l\geq k$ both of the squares
\[\xymatrixcolsep{7pc}
\xymatrix{
w_0^{s} \ar@{-}^{2^{-\text{ln}(s)-1}(1-2^{-l})}[r] \ar@{-}_{2^{-\text{ln}(s)-2}+\frac{1}{k+1}2^{-k-3}}[d] & z^{(t)}_l \ar@{-}^{2^{-\text{ln}(s)-l-2}+\frac{1}{k+1}2^{-k-l-1}}[d]
  \\
w^{s^\smallfrown j}_0 \ar@{-}_{2^{-\text{ln}(s)-2}(1-2^{-l})}[r] & z^{s^\smallfrown j}_l 
} \]
and
\[\xymatrixcolsep{4pc}
\xymatrix{
z_l^{s} \ar@{-}^{2^{-k-1}}[r] \ar@{-}_{2^{-\text{ln}(s)-l-2} +\frac{1}{k+1} 2^{-k-l-1}}[d] & z^{(t)}_{l+n} \ar@{-}^{2^{-\text{ln}(s)-l-n-2}+ \frac{1}{k+1}2^{-k-l-n-1}}[d]
  \\
z_l^{s^\smallfrown j} \ar@{-}_{2^{-k-1}}[r] & z^{s^\smallfrown j}_{l+n} 
}
\]
satisfy the triangle inequality.

Now, let
$$\ol{v}_k=\set{w_i^s, z_l^s \mid i\in [0,2^{-3}]\cap \Q, \ l\geq k, \ s\in 2^{\leq k}}\cup F_k$$
and 
$$\ol{u}_k=\set{w_i^s, z_l^s \mid i\in (0,2^{-3}]\cap \Q, \ l\geq k, \ s\in 2^{\leq k}}\cup F_k.$$
Clearly, both $\ol{v}_k$ and $\ol{u}_k$ are countable dense in $X_k$.
Put 
\[\p_k=(X_k,d_k,\ol{v}_k) \quad \text{and}\quad \s_k=(X_k,d_k,\ol{u}_k).\]

We will first argue that $ \phi_{\p_k}(F_k)=2+k$. Clearly $\phi_{\p_k}(F_k)\leq 2+k$ since $F_k$ contains $2^{k+1}$ elements. Moreover, it is straightforward to check that 
\[B_{d_k}(w_0^s,2^{\text{ln}(s)-1})=\ol{B}_{d_k}(w_0^s,2^{\text{ln}(s)-1})=\left(\bigsqcup_{t\in 2^{\leq k}, \ s\subseteq t}Y_k^t\right)\sqcup\set{ f_t\mid t\in 2^{k+1}, \ s\subseteq t}\]
for all $s\in 2^{\leq k}$. Therefore, we have
\[\ol{B}_{d_k}(w_0^{s^\smallfrown 0},2^{\text{ln}(s)-2}),\ol{B}_{d_k}(w_0^{s^\smallfrown 1},2^{\text{ln}(s)-2})\subseteq B_{d_k}(w_0^s,2^{\text{ln}(s)-1})\]
and
\[ \ol{B}_{d_k}(w_0^{s^\smallfrown 0},2^{\text{ln}(s)-2})\cap\ol{B}_{d_k}(w_0^{s^\smallfrown 1},2^{\text{ln}(s)-2})=\emptyset\]
for all $s\in 2^{<k}$. Since it also holds that
\[f_{s^\smallfrown 0},f_{s^\smallfrown 1}\in B_{d_k}(w_0^s,2^{\text{ln}(s)-1})\] for all $s\in 2^k$, we obtain $\phi_{\p_k}(F_k)= k+2$.

We will now argue that $\phi_{\s_k}(F_k)=2$. In this case it is clear that $\phi_{\s_k}(F_k)\geq 2$, as $f_{s^\smallfrown 0},f_{s^\smallfrown 1}\in F_k$ with $d_k(f_{s^\smallfrown 0},f_{s^\smallfrown 1})<2^{-1}$ for all $s\in 2^k$. Moreover, it is straightforward to check that if $x\in X_k\setminus \set{w_0^s\mid s\in 2^{\leq k}}$, $t\in 2^k$ and $m\in \omega$ such that 
\[f_{t^\smallfrown 0},f_{t^\smallfrown 1}\in B_{d_k}(x,2^{-m-1}),\]
then for all $s\in 2^{\leq k}$ there is $N_s\in \omega$ such that \[z_l^s\in B_{d_k}(x,2^{-m-1})\]
for all $l\geq N_s$. Hence, if $x,y\in \ol{u}_k$ and $n,m\in \omega$ satisfy that there is $s,t\in 2^k$ such that 
\[f_{s^\smallfrown 0},f_{s^\smallfrown 1}\in B_{d_k}(x,2^{-n-1}) \quad \text{and}\quad f_{t^\smallfrown 0},f_{t^\smallfrown 1}\in B_{d_k}(y,2^{-m-1}),\]
then $B_{d_k}(x,2^{-n-1})\cap B_{d_k}(y,2^{-m-1})\neq \emptyset$. We can therefore conclude that $\phi_{\s_k}(F_k)=2$.

Finally, let
\[X=[0,1]\sqcup\left(\bigsqcup_{2\leq k<\omega} X_k\right), \qquad F=\bigsqcup_{2\leq k<\omega} F_k, \qquad \ol{x}=\left([0,1]\cap \Q\right)\sqcup\left(\bigsqcup_{2\leq k<\omega}\ol{v}_k\right)\] 
and
\[ \ol{y}=\left((0,1]\cap \Q\right)\sqcup\left(\bigsqcup_{2\leq k<\omega}\ol{u}_k\right).\]
Moreover, define a complete metric $d$ on $X$ by letting
\begin{itemize}
\item[1)] $d_{|X_k}=2^{-1}d_k$ for all $2\leq k<\omega$.
\item[2)] $d(i,j)=|i-j|$ for all $i,j\in [0,1]$.
\item[3)] $d(0,w_0^{\emptyset,k})=2^{-2}$ for all $2\leq k<\omega$. Here $w_0^{\emptyset,k}$ denotes the element $w_0^\emptyset\in X_k$. 
\end{itemize}
Put 
\[\s=(X,d,\ol{x}) \qquad \text{and} \qquad \p=(X,d,\ol{y}).\]
Then, as
\[ \ol{B}_d(w_0^{\emptyset,k},2^{-2})\subseteq B_d(0,2^{-1})\]
and $\ol{B}_{d_k}(w_0^{\emptyset,k},2^{-1})=\ol{B}_d(w_0^{\emptyset,k},2^{-2})\cap X_k$
for all $2\leq k< \omega$, we must have $\phi_\p(F)=\omega+1$. Moreover, if $2\leq k<\omega$, $m\in \omega$ and $y\in X_k$ satisfy  that there is $t\in 2^k$ such that 
\[ f_{t^\smallfrown 0},f_{t^\smallfrown 1}\in B_{d}(y,2^{-m-1}),\]
then 
\[B_d(y,2^{-m-1})\setminus B_d(r,2^{-1})\neq \emptyset.\]
Therefore we can conclude that $\phi_\s(F)=2$.
\end{proof}

The next result gives a bound on the variation one can obtain by changing the dense sequence. 

\begin{thm}\label{boundchangseq} Let $\p=(X,d,\ol{x})$ and $\s=(X,d,\ol{y})$ be presentations of a Polish space $X$. Then
\[\phi_\p(F)\leq \omega\phi_\s(F)+2\]
for all $F\in F_{\aleph_0}(X)$.
\end{thm}
\begin{proof}
We will prove by induction on $\alpha<\omega_1$ that if $n\in D^{\omega\alpha +1}_\p(A_F^\p)$ with $n(1)>0$ and $\epsilon>0$, then there is $m\in D_\s^{1+\alpha}(A_F^\s)$ satisfying $d(x_{n(0)},y_{m(0)})<\epsilon$ and $m(1)=n(1)-1$. Note that this sufficient, since if $\phi_\s(F)=\alpha_0$ for some $\alpha_0<\omega_1$, then $D^{\alpha_0}_\s(A_F^\s)=\emptyset$ and hence $D_\p^{\omega\alpha_0 +1}(A_F^\p)\subseteq \omega\times \set{0}$. So $\phi_\p(F)\leq \omega\alpha_0+2$, as desired. 

For $\alpha=0$ the statement above is clearly true. 

Assume that the statement holds for some $\alpha<\omega_1$. Let $\epsilon>0$ and $n\in D_\p^{\omega(\alpha +1)+1}(A_F^\p)$ with $n(1)>0$. Then there are $u,v\in D_\p^{\omega(\alpha +1)}(A_F^\p)$ with $n\pc_\p u,v$ and $u\cw_\p v$.
Therefore for each $k\in \omega$ there exist $n^u_k,n_k^v\in D_\p^{\omega\alpha +1}(A_F^\p)$ such that $u\pc_\p n_k^u$, $v\pc_\p n_k^v$ and $n_k^u(1),n_k^v(1)> n(1)+k+1$. Next, as $d$ is complete and $u\cw_\p v$, we are in one of the following three cases
\begin{itemize}
\item[1)] there is $k\in \omega$ such that $d(x_{n_k^u(0)},x_{n_k^v(0)})> 2^{-k}$.
\item[2)] there are $k_0,k_1\in \omega$ such that $d(x_{n^u_{k_0}(0)},x_{n^u_{k_1}(0)})> 2^{-\min\set{k_0,k_1}}.$
\item[3)] there are $k_0,k_1\in \omega$ such that   $d(x_{n^v_{k_0}(0)},x_{n^v_{k_1}(0)})> 2^{-\min\set{k_0,k_1}}.$
\end{itemize} 
If we are in the first case, it follows by the induction hypothesis that there exist $m_0,m_1\in D_\s^{1+\alpha}(A_F^\s)$ with $m_0(1),m_1(1)\geq n(1)+ k+1$, $d(x_{n_k^u(0)},y_{m_0(0)})<2^{-n_k^u(1)-1}$ and $d(x_{n_k^v(0)},y_{m_1(0)})<2^{-n_k^v(1)-1}$. Then we have $m_0\cw_\s m_1$. Moreover, for any $m\in \omega^2$ with $d(x_{n(0)},y_{m(0)})< 2^{-n(1)-2}$ and $m(1)=n(1)-1$, we have $m\pc_\s m_0,m_1$. Therefore $m\in D^{1+\alpha +1}_\s(A_F^\s)$. The second and third case are handled analogously. 

Finally, assume that the statement holds for all $\beta<\lambda$ for some limit ordinal $\lambda<\omega_1$. Let $\epsilon>0$ and $n\in D^{\omega\lambda + 1}_\p(A_F^\p)$ be given. Then there are $\tilde{n}_\beta\in D^{\omega\beta +1}_\p(A_F^\p)$ for all $\beta<\lambda$ such that $n\pc_\p\tilde{n}_\beta$ and  $\tilde{n}_\beta(1)>n(1)+1$. Thus  for each $\beta<\lambda$ there is $m_\beta\in D^{1+\beta}_\s(A_F^\s)$ with $$d(x_{\tilde{n}_\beta(0)},y_{m_\beta(0)})<2^{-\tilde{n}_\beta(1)-1}$$ and $m_\beta(1)=\tilde{n}_\beta(1)-1$. Now, choose any $m\in \omega^2$ with $d(x_{n(0)},y_{m(0)})<2^{-n(1)-2}$ and $m(1)=n(1)-1$. Then $m\pc_\s m_\beta$ for all $\beta<\lambda$ and hence $m\in D^{1+\lambda}_\s(A_F^\s)$, as wanted. 
\end{proof}

\subsection{Change of metric}

Now we will investigate what happens when we change the complete metric. It is clear that the rank depends heavily on the chosen metric. But if for some $k\in \omega$ we have presentations $\p=(X,d,\ol{x})$ and $\s=(X,2^{-k}d,\ol{x})$ of a Polish space $X$, then $$\phi_\p(F)\leq \phi_\s(F) \leq \phi_\p(F)+k$$ for all $F\in F_{\aleph_0}(X)$. In general we have the following result bounding the change, when passing to an equivalent metric. Here we say that two compatible metrics $d,\delta$ on a Polish space $X$ are \textbf{equivalent} if there exists $N>0$ such that 
\[\frac{1}{N} d(x,y)\leq \delta(x,y)\leq Nd(x,y)\]
for all $x,y\in X$.
\begin{prop} Let $\p=(X,d,\ol{x})$ and $\s=(X,\delta,\ol{x})$ be presentations of a Polish space $X$ such that $d$ and $\delta$ are equivalent metrics. Then
\[\phi_\p(F)\leq \omega(\phi_\s(F)+1)\]
for all $F\in F_{\aleph_0}(X)$.
\end{prop} 

\begin{proof} First, since $d$ and $\delta$ are equivalent, we may fix $l\in \omega$ such that 
\[2^{-l}d(x,y)\leq \delta(x,y)\leq 2^ld(x,y)\]
for all $x,y\in X$.

By a similar argument as the one in the proof of Theorem \ref{boundchangseq}, one can prove the following.  
For all $\alpha<\omega_1$, if $n\in D^{\omega\alpha+1}_\p(A_F^\p)$ with $n(1)>l$, then  $\tilde{n}=(n(0),n(1)-l-1)\in D_\s^{1+\alpha}(A_F^\s)$.
 
This completes the proof, since if $\phi_\s(F)=\alpha_0$ for some $\alpha_0<\omega_1$, then $D^{\alpha_0}_\s(A_F^\s)=\emptyset$ and hence $D_\p^{\omega\alpha_0+1}(A_F^\p)\subseteq \omega\times \set{0,\ldots,l}$. Therefore we obtain $D_\p^{\omega(\alpha_0+1)}(A_F^\p)=\emptyset$, as wanted.

\end{proof}

Next we will see that there is no bound in general. 
For a presentation $\p=(X,d,\ol{x})$ of a Polish space $X$ and a homeomorphism $f\colon X\to X$ we define a presentation $\p_f=(X,d_f,\ol{y})$ of $X$ by $d_f(x,y)=d(f(x),f(y))$ for all $x,y\in X$ and $y_i=f^{-1}(x_i)$ for all $i \in \omega$. We then have the following connection between the two ranks $\phi_\p$ and $\phi_{\p_f}$.

\begin{prop}\label{changemetric} Let $\p$ be a presentation of a Polish space $X$ and $f\colon X\to X$ a homeomorphism. Then $\phi_{\p_f}(F)=\phi_\p(f(F))$ for all $F\in F_{\aleph_0}(X)$.
\end{prop}

\begin{proof} This follows directly from Lemma \ref{homoless} and the fact that 
\[x\in \ba_{\p_f}(n)\iff f(x)\in \bap(n)\]
for all $n\in \omf$.
\end{proof}

We will now use Proposition \ref{changemetric} to prove that a change in metric may result in arbitrarily countable change in the induced rank of a fixed closed discrete subset of $\omega^\omega$.

Let $S=\set{ s_i\in \omega^{<\omega} \mid i\in \omega}$ and $T=\set{t_i\in \omega^{<\omega}\mid i\in \omega}$ satisfy $\text{ln}(s_i)=\text{ln}(t_i)$, $s_i\perp s_j$, $t_i\perp t_j$ and $s_i\perp t_j$ for all $i,j\in \omega$ with $i\neq j$. We define the map $f_{S,T}\colon \omega^\omega\rightarrow \omega^\omega$ given by
\begin{equation*}
f_{S,T}(x)=\left\lbrace \begin{array}{cc}
t_i^\smallfrown z \quad &\text{if}\quad  x=s_i^\smallfrown z \ \text{for some} \ i\in \omega\ \text{and} \ z\in \omega^\omega\\
s_i^\smallfrown z \quad &\text{if}\quad  x=t_i^\smallfrown z \ \text{for some} \ i\in \omega\ \text{and} \ z\in \omega^\omega\\
x \quad &\text{otherwise} 
\end{array}\right. 
\end{equation*}
It is straightforward to check that $f_{S,T}$ is a homeomorphism. We call the pair $S,T$  \textbf{compatible sets} of initial segments, and we will say that $f_{S,T}$ is the induced \textbf{switch map}.

Below we consider the standard ultra-metric $\rho$ on $\omega^\omega$ given by
\[\rho(x,y)= \left\lbrace \begin{array}{cc}
0 \quad &\text{if} \quad x=y\\
2^{-\min\set{n\in \omega\mid x(n)\neq y(n)}-1} \quad &\text{if} \quad x\neq y
\end{array} \right. \]
for all $x,y\in \omega^\omega$.

\begin{thm}\label{strongdepmet} Let $\p=(\omega^\omega,\rho,\ol{x})$ be a presentation of $\omega^\omega$, where $\rho$ denotes the standard ultra-metric.
There is $F\in F_{\aleph_0}(\omega^\omega)$ and 
 homeomorphisms $f_\alpha \colon \omega^\omega\to \omega^\omega$ such that $\phi_\p(F)=2$ and $\phi_{\p_{f_\alpha}}(F)=\alpha$ for all $2\leq\alpha<\omega_1$. 
\end{thm}
\begin{proof}
First, let $F =\set{ \ol{n}\mid n\in \omega}$ and note that $\phi_\p(F)=2$. Here $\ol{n}=(n,n,n,\ldots )$ for $n\in \omega$.
We will now recursively construct compatible sets of initial segments $S_\alpha, T_\alpha$, such that the induced switch map $f_\alpha$ satisfies $\phi_{\p_{f_\alpha}}(F)=\alpha$ for all $2\leq \alpha <\omega_1$. 

For $\alpha=2$, let 
$S_2=\set{ s_i\in \omega^{<\omega} \mid i\in \omega}$ and $T_2=\set{t_i\in \omega^{<\omega}\mid i\in \omega},$
where $s_i=t_i=(i)$ for all $i\in \omega$. Clearly $S_2,T_2$ are compatible sets of initial segments.

Now assume that we have built the compatible sets of initial segments $S_\alpha=\set{s_i\in \omega^{<\omega}\mid i\in \omega}$ and $T_\alpha=\set{t_i\in \omega^{<\omega}\mid i\in \omega}$ 
for some $\alpha<\omega_1$. 
Then put \[ S_{\alpha+1}=\lbrace \tilde{s}_i\in \omega^{<\omega}\mid i\in \omega\rbrace, 
\]
where $\tilde{s}_{2i}$ and $\tilde{s}_{2i+1}$ is obtained from $s_i$ by replacing any occurrence of $n$ in $s_i(0)^\smallfrown s_i$ with $2n$ and $2n+1$, respectively, for all $n\in \omega$. Also put
\[ T_{\alpha+1}=\lbrace \tilde{t}_i\in \omega^{<\omega}\mid i\in \omega\rbrace, 
\]
where $\tilde{t}_{2i}=(0)^\smallfrown\tilde{t}^0_{2i}$ and $\tilde{t}_{2i+1}=(1)^\smallfrown \tilde{t}^0_{2i+1}$,  and $\tilde{t}_{2i}^0$ and $\tilde{t}^0_{2i+1}$ is obtained from $t_i$ by replacing any occurrence of $n$ in $t_i$ with $2n$ and $2n+1$, respectively, for all $n\in \omega$. It is  straightforward to check that if $S_\alpha,T_\alpha$ are compatible sets of initial segments, then $S_{\alpha+1},T_{\alpha+1}$ are compatible, as well.

Next let $\lambda<\omega_1$ be a limit ordinal and assume we have constructed compatible sets of initial segments $$S_\beta=\set{s_k^\beta\in \omega^{<\omega}\mid k\in \omega} \quad \text{and}\quad T_\beta=\set{t_k^\beta\in \omega^{<\omega}\mid k\in \omega}$$ for all $\beta<\lambda$.  Then let $(\beta_i)_i<\lambda$ be an increasing sequence such that $\bigcup_{i\in \omega}\beta_i =\lambda$ and fix an enumeration $(p_i)_i$ of the prime numbers. Then put 
\[ S_{\lambda}=\lbrace \tilde{s}_{i,k}\in \omega^{<\omega}\mid i,k\in \omega\rbrace, 
\]
where $\tilde{s}_{i,k}$  is obtained from $s_k^{\beta_i}$ by replacing any occurrence of $n$ in $s_k^{\beta_i}(0)^\smallfrown s_k^{\beta_i}$ with $p_i^{n+1}$ for all $n\in \omega$. Also put
\[ T_{\lambda}=\lbrace \tilde{t}_{i,k}\in \omega^{<\omega}\mid i,k\in \omega\rbrace, 
\]
where $\tilde{t}_{i,k}=(p_i)^\smallfrown \tilde{t}_{i,k}^0$ and $\tilde{t}_{i,k}^0$ is obtained from $t^{\beta_i}_k$ by replacing any occurrence of $n$ in $t_k^{\beta_i}$ by $p_i^{n+1}$ for all $n\in \omega$. Again, it is easy to check that $S_\lambda, T_\lambda$ are compatible sets of initial segments.

Finally, the construction of $S_\alpha,T_\alpha$ for $\alpha<\omega_1$  ensures that a straightforward induction argument shows that $\phi_{\p_\alpha}(F)=\phi_\p(f_\alpha(F))=\alpha$.

\end{proof}

 Note that if $d$ is an ultra-metric on a space $X$ and $f\colon X\to X$ is a homeomorphism, then $d_f$ is also an ultra-metric. Therefore it follows by Proposition \ref{indepdenseultra} that the presentations constructed in the theorem above are independent of the dense sequence.  Thus we obtain the following corollary.

\begin{cor} Let $\p=(\omega^\omega,\rho,\ol{x})$ be a presentation of $\omega^\omega$, where $\rho$ is the standard ultra-metric on $\omega^\omega$. There exists $F\in F_{\aleph_0}(\omega^\omega)$ such that for each $\alpha<\omega_1$ there is a presentation $\p_\alpha=(\omega^\omega,d_\alpha,\ol{x})$ satisfying $\phi_\p(F)+\alpha<\phi_{\p_\alpha}(F)$.
\end{cor}

Before we end this section, we will point out two more direct consequences of Theorem \ref{strongdepmet} that will become useful in the next section.

\begin{cor}\label{descritebairearbit} Let $\p=(\omega^\omega,\rho,\ol{y})$ be a presentation of $\omega^\omega$, where $\rho$ is the standard ultra-metric on $\omega^\omega$. For each $\alpha<\omega_1$ there is a discrete $F\in F_{\aleph_0}(\omega^\omega)$ with $\phi_\p(F)=\alpha$.
\end{cor}

\begin{proof}
Let $F$ and  $f_\alpha$ be as in the proof of Theorem \ref{strongdepmet} for each $2\leq \alpha<\omega_1$. One then has that $f_\alpha(F)\in F_{\aleph_0}(\omega^\omega)$  is discrete and satisfies $\phi_\p(f_\alpha(F))=\alpha$.  
Moreover, $\phi_\p(\emptyset)=0$ and $\phi_\p(\set{x})=1$ for all $x\in \omega^\omega$. 
\end{proof}

\begin{cor}\label{discretearbit} For each $1\leq\alpha<\omega_1$ there is a presentation $\p_\alpha=(\omega,d_\alpha, \omega)$ of the discrete Polish space $\omega$ such that ${\rm diam}_{d_\alpha}(\omega)<\infty$ and  $\phi_{\p_\alpha}(\omega)=\alpha$. 
\end{cor}

\begin{proof} For each $2\leq \alpha<\omega_1$ let $F$ and $f_\alpha$ be as in the proof of Theorem \ref{strongdepmet}, and fix a homeomorphism $g_\alpha\colon \omega\rightarrow f_\alpha(F)$. Let $d_\alpha$ denote the metric on $\omega$ given by $d_{\alpha}(i,j)=\rho(g_\alpha(i),g_\alpha(j))$ for all $i,j\in \omega$, and put $\p_\alpha=(\omega,d_\alpha,\omega)$ for all $2\leq \alpha<\omega_1$. Then we must have $\phi_{\p_\alpha}(\omega)=\alpha$. Moreover, as $\rho< 1$, we also have $d_\alpha<1$ for all $2\leq \alpha <\omega_1$.
 Clearly the presentation $\p=(\omega,d,\omega)$, where $d(i,j)=1$ for all $i,j\in \omega$ with $i\neq j$, satisfies $\phi_\p(\omega)=1$.

\end{proof}
Note that if $\p=(X,d,\ol{x})$ is a presentation of a Polish space $X$, then the metric $d_b$ given by
\[d_b(x,y)=\min\set{d(x,y),1}\]
for all $x,y\in X$ satisfies $\diam_{d_b}(X)<\infty$, and for $\s=(X,d_b,\ol{x})$ we have $\phi_\p=\phi_\s$.

\section{The relation to the Cantor-Bendixson rank}

In this section we will investigate how these ranks relate to the standard Cantor-Bendixson rank. It turns out,  that it depends heavily on how big the considered Polish space is. Recall that for any Polish space $X$ and $F\in F(X)$ we let $|F|_\cb$ denote the Cantor-Bendixson rank of $F$, and for all $\alpha<\omega_1$ we let $F^\alpha$ denote the iterated Cantor-Bendixson derivative of $F$. First we will argue that the ranks constructed in Section 2 refine the Cantor-Bendixson rank.

\begin{prop} Let $\p$ be any presentation of a Polish space $X$. For $F\in F(X)$ we have $A_{F^\alpha}^\p\subseteq D_\p^\alpha(A_F^\p)$ for all $\alpha <\omega_1$. In particular, $F^\alpha\subseteq F_{D_\p^\alpha(A_F^\p)}^\p $ for all $\alpha <\omega_1$. 
\end{prop}
\begin{proof}
We prove the result by induction on $\alpha<\omega_1$. The case $\alpha=0$ is trivial. Assume therefore that $A_{F^\alpha}^\p\subseteq D_\p^\alpha(A_F^\p)$ for some $\alpha<\omega_1$ and let $n\in A_{F^{\alpha+1}}^\p$. Then, by definition, $\bap(n)\cap F^{\alpha +1}\neq \emptyset$ and hence there must be $x,y\in \bap(n)\cap F^\alpha$ with $x\neq y$. Thus, we can find $k,l\in A_{F^\alpha}^\p\subseteq D_\p^\alpha(A_F^\p)$ such that $n\pc_\p k,l$ and $k\cw_\p l$.  Therefore we conclude
$n\in D_\p^{\alpha +1}(A_F^\p)$.

Next, assume $A_{F^\beta}^\p\subseteq D_\p^\beta(A_F^\p)$ for all $\beta <\lambda$ for some limit ordinal $\lambda<\omega_1$ and let $n\in A_{F^\lambda}^\p$. Since $A_{F^\lambda}^\p\subseteq A_{F^\beta}^\p\subseteq D_\p^\beta(A_F^\p)$ for all $\beta<\lambda$, we obtain that $n\in D_\p^\lambda(A_F^\p)$, as desired. 
\end{proof}

From this proposition we easily get the following result.
\begin{cor}\label{lowerbound} Let $\p$ be any presentation of a Polish space $X$. For all $F\in F_{\aleph_0}(X)$ we have $|F|_{\cb}\leq \phi_\p(F)$.
\end{cor}

\subsection{A characterization of compact spaces}

In this subsection we will characterize the compact Polish spaces in terms of how the ranks constructed in Section 2 relate to the Cantor-Bendixson rank.  More precisely we will prove the following theorem.
\begin{thm}\label{A} Let $X$ be a Polish space. The following are equivalent:
\begin{itemize}
\item[1)] $X$ is compact.
\item[2)] The family $(\phi_\p)_\p$, where $\p$ varies over all presentations of $X$, is uniformly bounded by the Cantor-Bendixson rank.
\item[3)] For any presentation $\p$ of $X$ we have $\phi_\p(F)<\omega|F |_{\cb}$ for all $F\in F_{\aleph_0}(X)$.
\end{itemize}
\end{thm}
 It is clear that $(3)\implies(2)$. Below we will show $(2)\implies (1)$ and $(1)\implies (3)$.
We will begin with the latter, where the goal is to prove that if $X$ is compact and $\p$ is a presentation of $X$, then $\phi_\p(F)< \omega|F|_{\cb}$ for all $F\in F_{\aleph_0}(X)$.
In order to obtain strict inequality we begin by proving that $\phi_\p(F)$ is a successor for any $F\in F_{\aleph_0}(X)$.  This will follow from the next couple of results.

\begin{prop}\label{independ} Let $\p$ be a presentation of any Polish space $X$ and let $F\in F(X)$. If $(m_i)_i,(n_i)_i\in [A_F^\p]_\p$ satisfy $\pi_\p((m_i)_i)=\pi_\p((n_i)_i)$, then 
\[(n_i)_i\in [D_\p^\alpha(A_F^\p)]_\p \iff (m_i)_i\in [D_\p^\alpha(A_F^\p)]_\p\]
for all $\alpha<\omega_1$. 
\end{prop}
\begin{proof}
For any $\p$-closed $A\subseteq \omf$, we have 
$x\in \pi_\p([A]_\p)$ if and only if for all $m\in \omf$ such that $x\in \bap(m)$ we have $m\in A$.
\end{proof}

\begin{prop}\label{bodysucc} Let $X$ be compact and $\p$ be a presentation of $X$. If $F\in F_{\aleph_0}(X)$ is non-empty and $\alpha<\omega_1$ is least such that $[D_\p^\alpha(A_F^\p)]_\p=\emptyset$, then $\alpha$ is a successor.
\end{prop}

\begin{proof} 
First, since $F$ is non-empty, we have $\alpha>0$.
Now assume for a  contradiction that $[D_\p^\beta(A_F^\p)]_\p\neq \emptyset $ for all $\beta <\lambda$ and $[D_\p^\lambda(A_F^\p)]_\p=\emptyset$ for some limit ordinal $\lambda<\omega_1$. Fix $(\beta_i)_{i}<\lambda$ such that $\beta_i\leq \beta_{i+1}$ for all $i\in \omega$ and $\bigcup_{i\in \omega}\beta_i=\lambda$. For each $i\in \omega$, choose $x_i\in  \pi_\p([D_\p^{\beta_i}(A_F^\p)]_\p)$.
By compactness of $X$, there is $x\in X$ and a subsequence $(x_{i_k})_k\subseteq (x_i)_i$ such that $x_{i_k}\rightarrow x$ as $k\rightarrow \infty$. 
 So, since  $\pi_p([D_\p^{\beta_i}(A_F^\p)]_\p)$ is closed and $$\pi_\p([D_\p^{\beta_{i+1}}(A_F^\p)]_\p)\subseteq \pi_\p([D_\p^{\beta_{i}}(A_F^\p)]_\p)$$ for all $i\in \omega$, we get that $x\in \pi_\p([D_\p^{\beta_i}(A_F^\p)]_\p)$ for all $i\in \omega$. Thus, by Proposition \ref{independ}, we obtain $x\in \pi_\p([D_\p^\lambda(A_F^\p)]_\p)$, which contradicts that $[D_\p^\lambda(A_F^\p)]_\p=\emptyset$.
\end{proof}

\begin{prop}\label{successor} Let $X$ be compact, $\p$ be a presentation of $X$ and assume that $[D_\p^\beta(A_F^\p)]_\p=\emptyset$ for some $\beta<\omega_1$. Then there is some $k\in \omega$ such that $D_\p^{\beta +k}(A_F^\p)=\emptyset$. 
\end{prop}

\begin{proof} Suppose for a  contradiction that this is not the case. Then for each $k\in\omega$ there is $n_k\in D_\p^\beta(A_F^\p)$ with $n_k(1)\geq k$.  Thus, by Lemma \ref{Konig}, we obtain $[D_\p^\beta(A_F^\p)]_\p\neq \emptyset$, a contradiction.
\end{proof}
Putting together Proposition \ref{bodysucc} and Proposition \ref{successor} we obtain the following result.

\begin{cor} Let $\p$ be a presentation of a compact Polish space $X$. Then for all $F\in F_{\aleph_0}(X)$ there is $\beta<\omega_1$ such that $\phi_\p(F)=\beta + 1$. 
\end{cor}

The next lemma is the main ingredient in the proof of $(1)\implies (3)$ in Theorem \ref{A}.


\begin{lem}\label{locallyupperbound} Let $X$ be a Polish space and $\p$ a presentation of $X$. If we have $F\in F_{\aleph_0}(X)$,  $\alpha <\omega_1$ and $n\in D_\p^{\omega\alpha}(A_F^\p)$ such that $\ol{\bap(n)\cap F}$ is compact, then there is $x\in \ol{\bap(n)\cap F}$ such that $x\in F^\alpha$.
\end{lem}

\begin{proof}
The statement is trivial for $\alpha=0$. 

Assume that the statement is true for some $\alpha<\omega_1$ and that $n\in D_\p^{\omega(\alpha +1)}(A_F^\p)$. For each $k\in \omega$ there must be $m_k\in D^{\omega \alpha +k}_\p( A_F^\p)$ such that $n\pc_\p m_k$. 
We will now recursively construct a sequence $(x_k)_k\in F\cap U_n$ such that $x_k\in F^\alpha$ for all $k\in \omega$ and $x_i\neq x_j$ whenever $i\neq j$.  First, since $m_0\in D^{ \omega\alpha}(A_F)$, it follows by the induction hypothesis that there is 
\[x_0\in \ol{\bap(m_0)\cap F}\subseteq \bap(n)\cap F\]
with $x_0\in F^\alpha$. Now assume we have constructed $x_0,\ldots , x_{k-1}$ for some $k>0$ satisfying the above. Then since $m_k\in D^{ \omega \alpha +k}(A_F)$ there is $(l_s)_{s\in 2^{\leq k}}\in D^{\omega\alpha }_\p(A_F^\p)$ such that
$l_\emptyset=m_k$, $l_s\pc_\p l_{s^\smallfrown 0},l_{s^\smallfrown 1}$ and $l_{s^\smallfrown 0}\cw_\p l_{s^\smallfrown 1}$ for all $s\in 2^{<k}$.
Now, since $k<2^k$, there must be $s\in 2^k$ such that $x_i\notin \cbap(l_s)$ for all $i<k$. Moreover, as $l_s\in D_\p^{\omega\alpha}(A_F^\p)$, it follows by the induction hypothesis that there is $$x_k\in \ol{\bap(l_s)\cap F}\subseteq \bap(n)\cap F$$ with $x_k\in F^\alpha$. By the choice of $l_s$, we obtain $x_k\neq x_i$ for all $i<k$. Continuing this way we obtain $(x_k)_k\in\bap(n)\cap F$ satisfying the above. By compactness of $X$, it follows that there is $x\in \ol{\bap(n)\cap F}$ with $x\in F^{\alpha+1}$. 

To finish the proof, let $\lambda<\omega_1$ be a limit ordinal and assume that the statement is true for all $\beta <\lambda$. Moreover, let $n\in D_\p^{\omega \lambda}( A_F^p)$. Fix $(\beta_i)_i<\lambda$ with $\beta_{i}\leq \beta_{i+1}$ for all $i\in \omega$ and $\bigcup_{i\in \omega}\beta_i=\lambda$.
Then,  by the induction hypothesis, we may for each $i\in \omega$ choose $x_i \in \ol{\bap(n)\cap F}$ with $x_i\in F^{\beta_i}$. Since $X$ is compact, there is $x\in \ol{\bap(n)\cap F}$ and a subsequence $(x_{i_k})_k\subseteq (x_i)_i$ such that $x_{k_i}\rightarrow x$ as $k\rightarrow \infty$. Then we must have $x\in F^{\beta_{i_k}}$ for all $k\in \omega$ and hence $x\in F^{\lambda}$. 
\end{proof}

We are now ready to conclude the proof of $(1)\implies (3)$ in Theorem \ref{A}.

\begin{thm}\label{upperbound} Let $X$ be a compact Polish space and $\p$ any presentation of $X$. Then $\phi_\p(F) < \omega|F|_{\cb}$ for all $F\in F_{\aleph_0}(X)$. 
\end{thm}

\begin{proof} Let $F\in F_{\aleph_0}(X)$.
First, as $\phi_\p(X)$ a successor, we may fix $\alpha <\omega_1$ such that $ \omega\alpha <\phi_\p(F) < \omega(\alpha + 1)$. Hence there exists $n\in D_\p^{\omega \alpha}( A_F^\p)$. So, by Lemma \ref{locallyupperbound}, there is some $x\in F^\alpha$. Therefore $|F|_{\cb}\geq \alpha +1$ and hence $\phi_\p(F)<  \omega|F|_{\cb}$.
\end{proof}

 Now we turn to the proof of $(2)\implies (1)$ in Theorem \ref{A}. Our strategy is to fix a non-compact Polish space and a discrete closed infinite subset. Then given $\alpha<\omega_1$ we will construct a presentation of the space, such that the associated rank of the infinite discrete subset is larger than $\alpha$. To construct this presentation we will need the following extension lemma for complete metrics, which is a ``complete'' version of Theorem 5 in \cite{Bing}. Our proof simply shows that the construction in \cite{Bing} can provide a complete metric under the assumptions below. 

\begin{lem}\label{extendmetric} Let $X$ be a completely metrizable space, $K\subseteq X$ a closed subset and $d_k$ a complete metric on $K$ with ${\rm diam}_{d_K}(K)<\infty$. Then there exists a complete metric $d$ on $X$ such that $d_{|K}=d_K$. 
\end{lem}
\begin{proof} First, by following the first part of the proof of \cite[Theorem 5]{Bing}, there is a complete metric $d'$ on $X$ satisfying $d'(x,y)\geq d_K(x,y)$ for all $x,y\in K$. 
Now, as in \cite[Theorem 5]{Bing}, define the metric $d$ on $X$ by
\[d(x,y)=\min\left\lbrace d'(x,y), \inf\set{ d'(x,z)+d_K(z,u)+d'(u,y)\mid u,z\in K}\right\rbrace.\]
It is easily seen that $d(x,y)=d_K(x,y)$ for all $x,y\in K$. So it suffices to prove that $d$ is complete. Let $(x_i)_i\in X$ be a Cauchy sequence with respect to $d$. We shall prove that there is $x\in X$ such that $x_i\rightarrow x$ as $i\rightarrow \infty$. Note that if $(x_i)_i$ is also Cauchy with respect to $d'$, then we would be done, since $d'$ is complete. So, assume that this is not the case. Then there is $\epsilon>0$ such that for each $N\in \omega$ there exist $i_N,j_N\geq N$ such that $d'(x_{i_N},x_{j_N})\geq \epsilon$. Next, for each $l\in \omega$ let $M_l\in \omega$ satisfy that 
$d(x_i,x_j)<2^{-l}\epsilon$ for all $i,j\geq M_l$. Moreover, assume that $M_{l+1}\geq M_l$ for all $l\in \omega$. Then, as $d'(x_{i_{M_l}},x_{j_{M_l}})\geq \epsilon$ for all $l\in \omega$, we must have that there is $u_l,z_l\in K$ such that
\[d'(x_{i_{M_l}}, u_l)+d_K(u_l,z_l)+d'(z_l,x_{j_{M_l}})<2^{-l}\epsilon\]
for all $l\in \omega$. We claim that $(u_l)_l\in K$ is a Cauchy sequence with respect to $d_K$. Indeed, we have
\begin{align*}
d_K(u_l,u_{l+1})&\leq d(u_l,x_{i_{M_l}})+d(x_{i_{M_l}},x_{i_{M_{l+1}}
})+d(x_{i_{M_{l+1}}}, u_{l+1})\\
&< d'(u_l,x_{i_{M_l}})+2^{-l}\epsilon
+d'(x_{i_{M_{l+1}}}, u_{l+1})\\
&<3\epsilon 2^{-l}.
\end{align*}
Thus, since $d_K$ is complete, there exists $u\in K$ such that $u_l\rightarrow u$ as $l\rightarrow \infty$. Finally, since we have $d(x_{i_{M_l}},u_l)<2^{-l}\epsilon$ for all $l\in \omega$, we also have $x_i\rightarrow u$ as $i\rightarrow \infty$.
\end{proof}

\begin{thm}Let $X$ be a non-compact Polish space and $F\in F_{\aleph_0}(X)$ an infinite discrete subset. For each $\alpha<\omega_1$ there is a presentation $\p_\alpha$ of $X$ such that $\phi_{\p_\alpha}(F)\geq\alpha$. 
\end{thm}

\begin{proof}
Let $\alpha<\omega_1$ be given and fix an enumeration $F=\set{y_i\mid i \in \omega}$. Moreover, by Corollary \ref{discretearbit}, we can fix an ultra-metric $d_F$ on $F$ that induces the discrete topology and such that the presentation $\s=(F,d_F,F)$ satisfies $\phi_\s(F)>3\alpha$ and $\diam_{d_F}(F)<\infty$. Using the presentation $\s$ of $F$, we will construct a presentation $\p$ of $X$ that satisfies $\phi_\p(F)\geq \alpha$. Let $\ol{x}$ be a countable dense sequence such that $x_{2k}=y_k$ for all $k\in \omega$. Applying Lemma \ref{extendmetric}, let $d$ be a complete metric on $X$ that extends $d_F$. Then put $\p=(X,d,\ol{x})$.  It now suffices to prove that $\phi_\s(F) \leq 3\phi_\p(F)+1$. 

First, note that for all $l,k,i,j\in \omega$ the triangle inequality implies that
\[ \bas(l,k+1) \pc_\s \bas(i,j) \implies \bap(2l,k)\pc_\p \bap(2i,j) \]
and
\[ \cbas(l,k)\cw_\s \cbas(i,j) \implies \cbap(2l,k+1)\cw_\p \cbap(2i,j+1).  \]
In the following, for any $n\in \omega^2$ with $n(1)>0$ we will let $\tilde{n}=(2n(0),n(1)-1)$.  We will prove by induction  on $\beta <\omega_1$ that if $n\in D^{3\beta}_\s(A_F^\s)$ and $n(1)>0$, then $\tilde{n}\in D^\beta_\p(A_F^\p)$. This is clearly true for $\beta=0$ and, by the induction hypothesis, when $\beta$ is a limit ordinal. We will therefore concentrate on the successor case. So assume that the statement holds for some $\beta<\omega_1$ and that $n\in D^{3\beta +3}_\s(A_F^\s)$ with $n(1)>0$. Then there exist $(n_s)_{s\in 2^{\leq 3}}\in D^{3\beta}_\s(A_F^\s)$ satisfying  $n=n_\emptyset$, $n_s\pc_\s n_{s^\smallfrown 0},n_{s^\smallfrown 1}$ and $ n_{s^\smallfrown 0}\cw_\s n_{s^\smallfrown 1}$ for all $s\in 2^{<3}$. Since $d_F$ is an ultra-metric, it follows from the above implications that there are   $s,t\in 2^3$ such that 
$\tilde{n}_s\cw_\p \tilde{n}_t$ and $\tilde{n}\pc_\p \tilde{n}_s,\tilde{n}_t$. Therefore by the induction hypothesis, we must have $\tilde{n}\in D^{\beta+1}_\p(A_F^\p)$, as wanted.

\end{proof}

\subsection{A characterization of $\sigma$-compact spaces}
In this subsection we prove a characterization of $\sigma$-compactness. To state the theorem, we will first recall another well-known rank that one can associate to Polish spaces, which measures how far from compact a $\sigma$-compact Polish space is.

For a Polish space $X$ with basis $(U_i)_i$ we let
\[X^*=\set{x\in X\mid (\forall i\in \omega) x\in U_i\implies \ol{U_i}\notin K(X) }.\]
Note that $X^*$ is closed, so we may recursively define the iterated derivatives $X_\alpha $ of $X$ for all $\alpha<\omega_1$ as follows:
\[ X_0=X, \qquad X_{\alpha+1}=( X_\alpha)^*\qquad \text{and} \qquad  X_\lambda =\bigcap_{\beta<\lambda} { X_\beta},\]
where $\alpha,\lambda<\omega_1$ and $\lambda$ a limit ordinal. There is a least ordinal $\alpha<\omega_1$ for which $X_\alpha=X_{\alpha+1}$. We call this ordinal the $K_\sigma$-rank of $X$ and denote it by $|X|_{K_\sigma}$. By construction it is clear that $X_{|X|_{K_\sigma}}=\emptyset$ if and only if $X$ is $K_\sigma$. Moreover, for each $\alpha < \omega_1$ we have that $X_\alpha\setminus X_{\alpha +1}$ is open and locally compact in $X_\alpha$. 
\\

We are now ready to state our characterization.
\begin{thm} Let $X$ be a Polish space. The following are equivalent:
\begin{itemize}
\item[1)] $X$ is $\sigma$-compact.
\item[2)] For some presentation $\p$ of $X$, there exists $f\colon \omega_1\rightarrow \omega_1$ such that $\phi_\p(F)\leq f(|F|_{\cb})$ for all $F\in F_{\aleph_0}(X)$.
\item[3)] For each presentation $\p$ of $X$, there exist ordinals $\alpha_\p,\beta_\p <\omega_1$ such that 
\[\phi_\p(F)\leq (\omega|F|_{\cb} + \alpha_\p)|X|_{K_\sigma} + \beta_\p,\]
for all $F\in F_{\aleph_0}(X)$. 
\end{itemize}
\end{thm}

\begin{proof}
It is clear that $(3)\implies (2)$, hence it suffices to prove $(2)\implies (1)$ and $(1)\implies (3)$.
\\

The implication $(2)\implies (1)$ is a consequence of the Boundedness Theorem. Indeed, let $\p$ be a presentation of $X$ and $f\colon \omega_1\rightarrow \omega_1$ be such that $\phi_\p(F)\leq f(|F|_{CB})$ for all $F\in F_{\aleph_0}(X)$. Then there is $\alpha_0\in \omega_1$ such that $\phi_\p(F)\leq \alpha_0$ whenever $F\in F_{\aleph_0}(X)$ is discrete. Now, assume for a contradiction that $X$ is not $\sigma$-compact. Then, since $X$ has $\omega^\omega$ as a closed subspace, we can obtain a co-analytic rank $\psi\colon F_{\aleph_0}(\omega^\omega)\rightarrow \omega_1$ for which $\psi(F)\leq \alpha_0$ for all discrete $F\in F_{\aleph_0}(\omega^\omega)$. Therefore, by the Boundedness Theorem,  all co-analytic ranks on $F_{\aleph_0}(\omega^\omega)$ are bounded on the discrete subsets, which contradicts Corollary \ref{descritebairearbit}.\\

For the implication $(1)\implies (3)$, assume that $X$ is $\sigma$-compact. Let $\p$ be a presentation of $X$ and  $\lambda=|X|_{K_\sigma}$. For each $\beta<\lambda$  define the subsets
\begingroup\makeatletter\def\f@size{11}\check@mathfonts
\begin{align*}
A_\beta^0&=\set{ n\in \omega^2\mid \bap(n)\cap X_\beta \neq \emptyset,  \bap(n)\cap X_{\beta+1}=\emptyset,    \ol{\bap(n)\cap X_\beta}\in K(X) } \\
A_\beta^1&=\set{ n\in \omega^2\mid \bap(n)\cap X_\beta \neq \emptyset,  \bap(n)\cap X_{\beta+1}=\emptyset, \ol{\bap(n)\cap X_\beta}\notin K(X) }
\end{align*}
\endgroup
and put $C=\omega^2\setminus (\bigcup_{\beta<\lambda}(A_\beta^0\cup A_\beta^1))$. For $i,j\in \set{0,1}$ and $\beta,\beta'<\lambda$ the following observations hold:
\begin{itemize}
\item[(1)] If $i\neq j$ or $\beta\neq \beta'$, then $A^i_\beta\cap A^j_{\beta'}=\emptyset$.
\item[(2)] If $n\in A_\beta^0$ and $m\in \omf$ satisfy $n\pc_\p m$,  then $m\in A^0_\gamma$ or $m\in A_{\gamma'}^1$ for some $\gamma\leq \beta$ or $\gamma'<\beta$.
\item[(3)] If $n\in A_\beta^1$ and $m\in \omf$ satisfy $n\pc_\p m$, then $m\in A^0_\gamma$ or $m\in A_{\gamma}^1$ for some $\gamma\leq \beta$.
\item[(4)] We have $[A_\beta^1]_\p=[C]_\p=\emptyset$.
\end{itemize}
From observation $(4)$ it follows that we may choose $\alpha_\p,\beta_\p<\omega_1$ such that $|C|_\p\leq \beta_\p$ and $|A_\beta^1|_\p\leq \alpha_\p$ for all $\beta<\lambda$. Note that if $\lambda$ is a successor, then $C=\emptyset$ and hence we may choose $\beta_\p=0$.

Now fix $F\in F_{\aleph_0}(X)$. We will argue that 
$$\phi_\p(F)\leq (\omega|F|_{\cb} + \alpha_{\p})\lambda + \beta_{\p}$$
by proving the following claims. 
\\
\\
\textbf{Claim 1:} Let $\beta<\lambda$. If $n\in D^{\omega \alpha}_\p(A_\beta^0\cap A_F^\p)$ for some $1\leq \alpha
 <\omega_1$, then there is $x\in \ol{\bap(n)\cap X_\beta}\cap F^\alpha$.\\
\\
\textit{Proof of Claim:} We will prove this claim by induction on $\alpha$.  First, assume that $n\in D_\p^\omega(A_\beta^0\cap A_F^\p)$. Then there is $(m_i)_i\in A_\beta^0\cap A_F$ and $(x_i)_i,(y_i)_i\in X$ such that for all $i,j\in \omega$ we have
\begin{itemize}
\item[(a)] $n\pc_\p m_i$ and $m_i(1)\geq n(1)+i$ 
\item[(b)] $x_i\in \bap(m_i)\cap X_\beta$ and $y_i\in \bap(m_i)\cap F$
\item[(c)]  $x_i\neq x_j$ and $y_i\neq y_j$ whenever $i\neq j$
\end{itemize}
Since $\ol{\bap(n)\cap X_\beta}$ is compact by definition of $A_\beta^0$, there must exist $x\in \ol{\bap(n)\cap X_\beta}$ and a subsequence $(x_{i_j})_j\subseteq (x_i)_i$ such that $x_{i_j}\rightarrow x$ as $j\rightarrow \infty$. Now, since $x_i,y_i\in \bap(m_i)$ for all $i\in \omega$ and $\diam(\bap(m_i))\rightarrow 0$ as $i\rightarrow \infty$, we must have $y_{i_j}\rightarrow x$ as $j\rightarrow \infty$, as well. Therefore, as $(y_{i_j})_j\in F$ and $F$ is closed, we must have that $x\in F^1$, as desired. 

The proof of the sucessor and the limit case can now be done as in the proof of Lemma \ref{locallyupperbound}.\begin{flushright}
$\lozenge$
\end{flushright}
 Before the next claim, we note that for all $\beta<\lambda$ an easy induction argument on $\eta<\omega_1$ shows, that observations (2),(3) implies that  for $n\in A_\beta^0$ we have
\[n\in D^\eta_\p\left(A_F^\p\setminus (\bigcup_{\gamma< \beta}(A_\gamma^0\cup A_\gamma^1))\right) \implies n\in D^\eta_\p(A_F^\p\cap A_\beta^0)\]
and
for $n\in A_\beta^1$ we have
\[n\in D^\eta_\p\left(A_F^\p\setminus (A_\beta^0\cup\bigcup_{\gamma< \beta}(A_\gamma^0\cup A_\gamma^1))\right) \implies n\in D^\eta_\p(A_F^\p\cap A_\beta^1)\]
for all $\eta<\omega_1$. From these implications we will obtain the next claim.\\
\\
\textbf{Claim 2:} We have
\begin{align*}
D^{(\omega|F|_\cb + \alpha_\p)\beta + \omega|F|_\cb}_\p(A_F^\p)&\subseteq A_F^\p\setminus (A_\beta^0\cup\bigcup_{\gamma< \beta}(A_\gamma^0\cup A_\gamma^1))\\
D^{(\omega|F|_\cb + \alpha_\p)(\beta+1)}_\p(A_F^\p)&\subseteq A_F^\p\setminus \bigcup_{\gamma\leq \beta}(A_\gamma^0\cup A_\gamma^1)
\end{align*}
for all $\beta<\lambda$.
\\
\\
\textit{Proof of Claim 2:} 
First we consider the case $\beta=0$.  If $n\in D_\p^{\omega|F|_\cb}(A_F^\p)$, then it follows by Claim 1  that $n\notin A_0^0$. By the above implications, we also obtain that if
\[n\in D_\p^{\omega|F|_\cb +\alpha_\p}(A_F^\p)=D^{\alpha_\p}_\p(D^{\omega|F|_\cb}_\p(A_F^\p))\subseteq D^{\alpha_\p}_\p(A_F^\p\setminus A_0^0),\]
then $n\notin A_0^1$, as $|A_0^1|_\p\leq \alpha_\p$.

Next, assume that the claim holds for some $\beta<\lambda$. Then we have
\begin{align*}
D^{(\omega|F|_\cb + \alpha_\p)(\beta+1)+ \omega|F|_\cb}_\p(A_F^\p)&=D^{\omega|F|_\cb}_\p\left(D_\p^{(\omega|F|_\cb + \alpha_\p)(\beta+1)}(A_F^\p)\right)\\
&\subseteq D^{\omega|F|_\cb}_\p\left(A_F^\p\setminus \bigcup_{\gamma\leq \beta}(A_\gamma^0\cup A_\gamma^1)\right).
\end{align*} 
So,  if $n\in D^{(\omega|F|_\cb + \alpha_\p)(\beta+1)+\omega|F|_\cb}_\p(A_F^\p)$, then $n\notin \bigcup_{\gamma< \beta+1}(A_\gamma^0\cup A_\gamma^1)$ and by Claim 1 and the implications above we have that $n\notin A_{\beta+1}^0$. Similarly one obtains
\[D^{(\omega|F|_\cb + \alpha_\p)((\beta+1)+1)}_\p(A_F^\p)\subseteq A_F^\p\setminus \bigcup_{\gamma\leq \beta+1}(A_\gamma^0\cup A_\gamma^1).\]
Finally assume that the claim holds for all $\beta<\xi$ for some limit $\xi<\omega_1$. Then we have
\begin{align*}
D^{(\omega|F|_\cb + \alpha_\p)\xi+ \omega|F|_\cb}_\p(A_F^\p)&=D^{\omega|F|_\cb}_\p\left(\bigcap_{\beta<\xi} D_\p^{(\omega|F|_\cb + \alpha_\p)(\beta+1)}(A_F^\p)\right)\\
&\subseteq D^{\omega|F|_\cb}_\p\left(A_F^\p\setminus(\bigcup_{\gamma< \xi}(A_\gamma^0\cup A_\gamma^1)\right).
\end{align*} 
So $n\in D^{(\omega|F|_\cb + \alpha_\p)\xi+ \omega|F|_\cb}_\p(A_F^\p)$ implies $n\notin A_\xi^0\cup\bigcup_{\gamma< \xi}(A_\gamma^0\cup A_\gamma^1)$, as desired. The other inclusion is obtained similarly.
\begin{flushright}
$\lozenge$
\end{flushright}

To finish the proof, note that Claim 2 implies that 
\[D_\p^{(\omega|F|_\cb +\alpha_\p)\lambda}(A_F^\p)\subseteq C\]
and therefore we may conclude that $D_\p^{(\omega|F|_\cb +\alpha_\p)\lambda +\beta_\p}(A_F^\p)=\emptyset$.

\end{proof}

\bigskip

 Department of Mathematical Sciences

 University of Copenhagen

 Universitetsparken 5

 DK-2100 Copenhagen

\textsf{vibquo@math.ku.dk}

\end{document}